\documentclass[10pt,a4paper]{amsart}

\usepackage{amsthm}
\usepackage{amsfonts}
\usepackage{amsmath}
\usepackage[english]{babel}

\usepackage{faktor}
\usepackage{etex}
\usepackage{hyperref}
\usepackage{multimedia}
\usepackage[all]{xy}
\usepackage{amssymb}
\usepackage{graphicx,textpos}
\usepackage[dvipsnames]{xcolor}
\usepackage{helvet}
\usepackage{stmaryrd}
\usepackage{enumerate}
\usepackage{todonotes}
\usepackage{pdfsync}
\usepackage{dsfont}
\usepackage[shortcuts]{extdash}
\usepackage[all]{xy}
  \newdir{ >}{{}*!/-9pt/@{>}}
 \newcommand{\suchthat}{\;\ifnum\currentgrouptype=16 \middle\fi|\;}

\newcommand{\R}{\mathbb{R}}
\newcommand{\Q}{\mathbb{Q}}
\newcommand{\Z}{\mathbb{Z}}
\newcommand{\C}{\mathbb{C}}

\newcommand{\rst}[1]{\ensuremath{{\mathbin\mid}\raise-.5ex\hbox{$#1$}}}
\newcommand{\lie}{\mathfrak{g}}

\newcommand{\lien}{\mathfrak{n}}

\newcommand{\liem}{\mathfrak{m}}

\DeclareMathOperator{\Gal}{Gal}
\DeclareMathOperator{\GL}{GL}

\DeclareMathOperator{\Aut}{Aut}
\DeclareMathOperator{\Aff}{Aff}

\makeatletter
\newtheorem*{rep@theorem}{\rep@title}
\newcommand{\newreptheorem}[2]{%
\newenvironment{rep#1}[1]{%
 \def\rep@title{#2 \ref{##1}}%
 \begin{rep@theorem}}%
 {\end{rep@theorem}}}
\makeatother

\author{Jonas Der\'e}
\address{KU Leuven Kulak, E. Sabbelaan 53, BE-8500 Kortrijk, Belgium}
\email{jonas.dere@kuleuven-kulak.be}
\thanks{The author is supported by a Ph.D.~fellowship of the Research Foundation -- Flanders (FWO). Research supported by the research Fund of the KU Leuven}
\title{\bf Gradings on Lie algebras with applications to infra-nilmanifolds}
\date{\today}

\newtheorem{Def}{Definition}[section]
\newtheorem*{Ex}{Example}
\newtheorem{Cor}[Def]{Corollary}
\newtheorem{Thm}[Def]{Theorem}
\newtheorem{Prop}[Def]{Proposition}
\newtheorem{Lem}[Def]{Lemma}

\newtheorem*{Thm*}{Theorem}
\newtheorem*{Prop*}{Proposition}
\newtheorem*{Lem*}{Lemma}

\newreptheorem{theorem}{Theorem}
\begin{document}

\begin{abstract}
In this paper, we study positive as well as non-negative and non-trivial gradings on finite dimensional Lie algebras. The existence of such a grading on a Lie algebra is invariant under taking field extensions, a result very recently obtained by Y. Cornulier and we give a different proof of this fact. Similarly, we prove that given a grading of one of these types and a finite group of automorphisms, there always exist a grading of the same type which is preserved by this group. From these results we conclude that the existence of an expanding map or a non-trivial self-cover on an infra-nilmanifold depends only on the covering Lie group. Another application is the construction of a nilmanifold admitting an Anosov diffeomorphisms but no non-trivial self-covers and in particular no expanding maps, which is the first known example of this type.
\end{abstract}

\maketitle

Let $E\subseteq \C$ be a subfield of the complex numbers and $\lien$ a finite dimensional Lie algebra over $E$. A grading of the Lie algebra $\lien$ is a decomposition of $\lien$ as a direct sum $$\displaystyle \lien = \bigoplus_{i \in \Z} \lien_i$$ of subspaces $\lien_i \subseteq \lien$ such that $[\lien_i,\lien_j] \subseteq \lien_{i+j}$ for all $i,j \in \Z$. We call the grading positive if $\lien_i = 0$ for all $i \leq 0$ and non-negative if $\lien_i = 0$ for all $i < 0$. Every Lie algebra has a non-negative grading given by $\lien = \lien_0$ and we call this the trivial grading of $\lien$. An automorphism $\varphi \in \Aut(\lien)$ preserves the grading if $\varphi(\lien_i) = \lien_i$ for all $i \in \Z$.

\smallskip

There is a strong relation between positive gradings and expanding automorphisms of Lie algebras. An automorphism $\varphi \in \Aut(\lien)$ is called expanding if $\vert \lambda \vert > 1$ for all eigenvalues $\lambda$ of $\varphi$. Let $\lien$ be a Lie algebra over $E$ with a positive grading $\lien = \bigoplus_{i \in \Z} \lien_i$. Given a $\mu \in E$ with $\vert \mu \vert > 1$, there exists an expanding automorphism $\varphi$ of $\lien$ which maps an element $x \in \lien_i$ to $\mu^i x$. Moreover, if $\psi$ is an automorphism of $\lien$ which preserves the grading, then $\psi$ will commute with the expanding automorphism $\varphi$. 

In \cite[Theorem 3.4]{dd14-1} it is shown that this construction also works the other way around. Given an expanding automorphism $\varphi$ we find a positive grading such that every automorphism $\psi$ commuting with $\varphi$ preserves the positive grading. The result in \cite{dd14-1} was only stated in the case where $E = \Q$ but in fact the proof works for any field $E$. 

There is a similar relation between non-trivial and non-negative gradings and partially expanding automorphisms. We call an automorphism partially expanding if for every eigenvalue $\lambda$ of $\varphi$ we have $\lambda = 1$ or $\vert \lambda \vert >1$ and there is at least one eigenvalue $\lambda \neq 1$. Therefore, the study of these two types of gradings on Lie algebras preserved by a given automorphism is equivalent to the study of (partially) expanding automorphisms which commute with this automorphism.

\smallskip

In this paper, we use several notations and results about linear algebraic groups. We refer the reader to \cite{bore91-1,hoch81-1,hump81-1} for more background about these groups. Note that the automorphism group of a Lie algebra is an example of a linear algebraic group. In the first section, we show that the existence of a positive grading on a Lie algebra is invariant under taking field extensions, see Theorem \ref{fieldalg}. The second main result, namely Theorem \ref{classexp}, shows that given a finite group of automorphisms, there always exists a grading preserved by this group. The same results also hold for non-trivial and non-negative gradings on Lie algebras. In the third section, we combine these main results about Lie algebras with \cite[Theorem 4.2.]{dd14-1} to get the following classification of infra-nilmanifolds admitting an expanding map:
\begin{reptheorem}{classexp}
Let $\Gamma \backslash G$ be an infra-nilmanifold modeled on a Lie group $G$ with corresponding Lie algebra $\lie$. Then the following are equivalent:
\begin{enumerate}[$(1)$]
\item the infra-nilmanifold $\Gamma \backslash G$ admits an expanding map;
\item the Lie algebra $\lie$ has a positive grading;
\item the Lie algebra $\lie$ has an expanding automorphism.
\end{enumerate}  
\end{reptheorem}
\noindent Again, there is a similar classification for non-trivial self-covers of $\Gamma \backslash G$, corresponding to non-trivial and non-negative gradings. As a consequence of this theorem we construct a nilmanifold admitting an Anosov diffeomorphism but no expanding map in the last section.

\smallskip

\noindent Some of these results were proved independently and by different methods in \cite{corn14-1} by Y. Cornulier. More detailed references to the work of Y. Cornulier are given throughout paper.

\section{Gradings under field extensions}

From now on, $K$ will always denote a subfield of the complex numbers $\C$. If $G \le \GL(n,\C)$ is a linear algebraic $K$-group, we denote by $G(K) = G \cap \GL(n,K)$ the subgroup of $K$-rational points in $G$. An element $x \in G$ is expanding if for all the eigenvalues $\lambda$ of $x$, we have $\vert \lambda \vert >1$. We call an automorphism partially expanding if for every eigenvalue $\lambda$ of $x$ we have $\lambda = 1$ or $\vert \lambda \vert >1$ and $x$ has at least one eigenvalue $\neq 1$. 

In this section we show that the existence of a positive grading is invariant under field extensions. First, we give a proof in the more general case of linear algebraic groups:

\begin{Thm}
\label{field}
Let $K \subseteq L \subseteq \C$ be field extensions and $G$ a linear algebraic $K$-group. Then $G(K)$ has an expanding element if and only if $G(L)$ has an expanding element.
\end{Thm}

\begin{proof}
We have a natural inclusion $G(K) \subseteq G(L)$, so if $G(K)$ has an expanding element, also $G(L)$ has an expanding element. For the other implication, it is sufficient to prove it in the case where $L = \C$. Since every power of an expanding element is again expanding, we can also assume that the group $G$ is connected. 

Let $x \in G(\C) = G$ be an expanding element. By the multiplicative Jordan decomposition, we can assume that $x$ is semisimple. Every semisimple element of $G$ lies in a maximal torus, so the existence of an expanding element is equivalent to the existence of an expanding element in a maximal torus. Since all maximal tori are conjugate and $G$ also contains a maximal torus defined over $K$, we can assume that $G$ is a $K$-torus. 

From \cite[Section 8.15]{bore91-1} it follows that every $K$-torus $G$ can be written as $G = G_a G_d$ where $G_a$ is an anisotropic subtorus and $G_d$ is a $K$-split subtorus. Since every $K$-split torus is conjugated over $\GL(n,K)$ to a $K$-closed subgroup of $D(n,\C)$, we can assume that $G_d$ is a subgroup of $D(n,\C)$. Let $x \in G$ be an expanding element and write $x = y z$ with $y \in G_a$ and $z \in G_d$. We show that $z$ is also expanding and hence that the $K$-split torus $G_d$ contains an expanding element. If not, we can write $z$ up to permutation of the eigenvalues and thus up to conjugation by an element of $\GL(n,\Q) \subseteq \GL(n,K)$ as $$z = \begin{pmatrix}Z_1 & 0 \\0 & Z_2 \end{pmatrix}$$ where $Z_1$ is a diagonal matrix with only eigenvalues $\leq 1$ in absolute value and $Z_2$ a diagonal matrix with only eigenvalues $>1$ in absolute value. Since the torus $G_a$ commutes with $z$, every element $t \in G_a$ can be written as $$t = \begin{pmatrix}A_t & 0 \\0 & B_t \end{pmatrix}$$ for some invertible matrices $A_t$ and $B_t$. Consider the character $\chi: G_a \to \C^\ast$ with $\chi(t) = \det(A_t)$. This morphism is clearly defined over $K$ and hence must be trivial. In particular, for the element $y$ we have $\det(A_y) = 1$. This would imply that $x$ is not expanding, a contradiction, and we conclude that $G_d$ contains an expanding element.

So we are left to prove the theorem in the case where $G$ is a $K$-split torus, or equivalently, in the case of a $K$-closed subgroup of $D(n,\C)$. Note that the expanding elements form an open subset of $G(\R)$ and $G(\C)$ for the Euclidean topology. The case $K \not\subseteq \R$ is then immediate since $G(K)$ forms a dense subset of $G$ for the Euclidean topology. In the case $K \subseteq \R$, we only have that $G(K)$ is a dense subset of $G(\R)$ for the Euclidean topology and therefore it suffices to show that $G(\R)$ has an expanding element. 

From \cite[Section 8.2]{bore91-1} it follows that $G$ is defined by character equations, so as the intersection of kernels of characters. Every character $D(n,\C) \to \C^\ast$ is of the form $$ \begin{pmatrix}\lambda_1 & 0 & \dots & 0 \\0 & \lambda_2 & \dots & 0 \\ \vdots & \vdots & \ddots & \vdots\\0 & 0 & \dots & \lambda_n\end{pmatrix} \mapsto \lambda_1^{k_1} \lambda_2^{k_2} \ldots \lambda_n^{k_n}$$ for some $k_i \in \Z$. This implies that if an element $$\begin{pmatrix}\lambda_1 & 0 & \dots & 0 \\0 & \lambda_2 & \dots & 0 \\ \vdots & \vdots & \ddots & \vdots\\0 & 0 & \dots & \lambda_n\end{pmatrix} \in G, $$ then also $$\begin{pmatrix} \vert \lambda_1 \vert & 0 & \dots & 0 \\0 &\vert \lambda_2 \vert & \dots & 0 \\ \vdots & \vdots & \ddots & \vdots\\0 & 0 & \dots & \vert \lambda_n \vert \end{pmatrix} \in G,$$ since the latter will also satisfy the same character equations. By applying this to an expanding element of $G$, we get an expanding element of $G(\R)$ and this finishes the proof.
\end{proof}

Let $E \subseteq \C$ be any field and $\lien$ a nilpotent Lie algebra over the field $E$. If $E \subseteq F$ is a field extension of $E$, then we can construct the Lie algebra $\lien^F = F \otimes_E \lien$ over the field $F$. The standard example will be the case where $F = \C$ and we call $\lien^\C$ the complexification of the Lie algebra $\lien$. The automorphism group $G = \Aut(\lien^\C)$ is a linear algebraic $E$-group and we have that $\Aut(\lien) = G (E)$.

By applying Theorem \ref{field} to the automorphism group of a Lie algebra, we have the following consequence:

\begin{Thm}
\label{fieldalg}
Let $E \subseteq F \subseteq \C$ be field extensions and $\lien$ a Lie algebra over the field $E$. Then $\lien$ admits an expanding automorphism if and only if $\lien^F$ admits an expanding automorphism. Equivalently, the Lie algebra $\lien$ admits a positive grading if and only if the Lie algebra $\lien^F$ admits a positive grading.
\end{Thm}

\noindent At this point, we want to remark here that Yves Cornulier presented the first proof of this theorem and this over all fields of characteristic 0 in \cite[Theorem 1.4]{corn14-1}. The approach we present here was developed independently from \cite[Theorem 1.4]{corn14-1} and uses different methods. 

\smallskip

More or less the same proof also works for partially expanding elements instead of expanding elements. One difference in the case of partially expanding elements is that the set of all partially expanding elements does not form an open subset of $G(\R)$ or $G(\C)$ for the Euclidean topology. If $G$ is a $K$-closed subgroup of $D(n,\C)$ containing a partially expanding element with the first $k$ eigenvalues equal to $1$, then we restrict to the subtorus $$\tilde{G} = \left\{  \begin{pmatrix}\lambda_1 & 0 & \dots & 0 \\0 & \lambda_2 & \dots & 0 \\ \vdots & \vdots & \ddots & \vdots\\0 & 0 & \dots & \lambda_n \end{pmatrix} \in G  \mid   \lambda_1 = \ldots = \lambda_k = 1 \right\}< G$$ and the proof works for this subtorus $\tilde{G}$.

If we write a partially expanding automorphism of $\tilde{G}$ as $x = y z$ with $y \in G_a$ and $z \in G_d$ for an anisotropic torus $G_a$ and a $K$-split torus $G_d$, we can show just as in Theorem \ref{field} that $z$ has no eigenvalues of absolute value $< 1$. This implies that the $K$-split torus $G_d$ has a partially expanding element. In particular we get the following result:

\begin{Thm}
\label{fieldalg2}
Let $\lien$ be a Lie algebra over a field $E \subseteq \C$ and $E \subseteq F \subseteq \C$ a field extension. Then $\lien$ admits a partially expanding automorphism if and only if $\lien^F$ admits a partially expanding automorphism. Equivalently, $\lien$ admits a non-trivial and non-negative grading if and only if $\lien^F$ admits a non-trivial and non-negative grading.
\end{Thm}

%

\section{Gradings preserved by automorphisms}

In this section, we study gradings preserved by automorphisms which lie in a reductive subgroup of the automorphism group. Note that all finite subgroups and diagonalizable groups of automorphisms satisfy this property, since they only contain semisimple elements. Again, we first consider the more general case of linear algebraic groups:

\begin{Thm}
\label{holonomy}
Let $G$ be a linear algebraic $K$-group with an expanding element and $H \subseteq G(K)$ a subgroup contained in a reductive subgroup of $G$. Then there exists an expanding element of $G(K)$ which commutes with every element of $H$.
\end{Thm}
\begin{proof}
The elements of $G$ which commute with every element of $H$ form a $K$-closed subgroup of $G$, so because of Theorem \ref{field} it suffices to show that $G$ contains an expanding element which commutes with every element of $H$. 

Since $H$ is a subgroup of a reductive subgroup of $G$, it follows from  From \cite[Chapter VIII]{hoch81-1} that there exists a Levi factor $L$ of $G$ such that $H \subseteq L$. Every expanding semisimple element of $G$ also lies in a Levi factor. Because all Levi factors of $G$ are conjugate, we know that if the group $G$ has an expanding element, the Levi factor $L$ (and therefore also $L^0$) has an expanding element. 

The group $L^0$ is by definition reductive and thus $L^0$ can be written as $$L^0 = Z \cdot \mathcal{D} L^0$$ where $Z = Z(L)^0$ is the identity component of the center of $L$ and $\mathfrak{D} L^0 = [L^0, L^0]$ is the commutator subgroup, see \cite[Section 14.2]{bore91-1}. Take $x$ an expanding element of $L^0$, then we can write it as $x = z y$ with $y \in \mathcal{D} L^0$ and $z \in Z$. Just as above in Theorem \ref{field}, we claim that $z$ is an expanding element.  If $z$ is not expanding, we can write $z$ up to conjugation in $\GL(n,\C)$ as $$z = \begin{pmatrix}Z_1 & 0 \\0 & Z_2 \end{pmatrix}$$ where $Z_1$ is an invertible matrix with only eigenvalues $\leq 1$ in absolute value and $Z_2$ an invertible matrix with only eigenvalues $>1$ in absolute value. Since $z$ lies in the center of $L^0$, every element $t \in L^0$ is of the form $$t = \begin{pmatrix}A_t & 0 \\0 & B_t \end{pmatrix}$$ for some invertible matrices $A_t$ and $B_t$. Consider now the character $\chi: L^0 \to \C^\ast$ given by $\chi( t ) = \det(A_t)$. Since $y \in \mathcal{D} L^0$, we have $\chi(y) = 1$ and this is a contradiction since $x = z y $ is expanding. We deduce that the connected center of $L^0$ contains an expanding element. 

The group $L^0$ has finite index in $L$, hence there exist elements $l_1, \ldots, l_k \in L$ such that $L = l_1 L^0 \sqcup \ldots \sqcup l_k L^0$. The subgroup $Z$ is a characteristic subgroup in $L$, so we have $l_i Z l_i^{-1} = Z$ for all $1 \leq i \leq k$. Fix an expanding element $z \in Z$, then the element $$z_0 = \prod_{1 \leq i \leq k} l_i z l_i^{-1} \in Z$$ commutes with every element $l_i$ with $1 \leq i \leq k$ and therefore also with every element of $L$. The group $H$ is a subgroup of $L$ and thus $z_0$ commutes with every element of $H$. Since all the elements $l_i z l_i^{-1} \in Z$ are expanding and commute, it follows that $z_0$ is expanding. This finishes the proof. \end{proof}

By applying this theorem to the linear algebraic group $\Aut(\lien^\C)$ of a rational Lie algebra $\lien$ and taking $H$ a finite subgroup, we get:

\begin{Thm}
\label{holonomyalg}
Let $\lien$ be a rational Lie algebra which admits an expanding automorphism and $H \subseteq \Aut(\lien)$ a finite subgroup. Then there exists an expanding automorphism of $\lien$ which commutes with every element of $H$. Equivalently, if $\lien$ admits a positive grading, it also admits a positive grading which is preserved by $H$.
\end{Thm}

\noindent The same result holds for partially expanding maps with the same proof:

\begin{Thm}
\label{holonomyalg2}
Let $\lien$ be a rational Lie algebra which admits a partially expanding automorphism and $H \subseteq \Aut(\lien)$ a finite subgroup. Then there exists a partially expanding automorphism of $\lien$ which commutes with every element of $H$. Equivalently, if $\lien$ admits a non-negative and non-trivial grading, it also admits a non-negative and non-trivial grading which is preserved by $H$.
\end{Thm}

\noindent These results were proved independently and by different methods in \cite[Corollary 3.26]{corn14-1}. 

\section{Applications to infra-nilmanifolds}

In this section we apply the main results to expanding maps and non-trivial self-covers on infra-nilmanifolds. First we fix some notation about infra-nilmanifolds, more background can be found in \cite{deki96-1,sega83-1}.

Let $G$ be a connected and simply connected nilpotent Lie group and $\Aut(G)$ the group of continuous automorphisms of $G$. Define the affine group $\Aff(G)$ as the semi-direct product $G \rtimes \Aut(G)$ which acts on $G$ in a natural way. An almost-Bieberbach group $\Gamma \subseteq G \rtimes H$, where $H$ is any finite subgroup of $\Aut(G)$, is a discrete, torsion-free subgroup such that the quotient $\Gamma \backslash G$ is compact. We will always assume that $H$ is the projection of $\Gamma$ on its second component and we call $H$ the holonomy group of $\Gamma$. The quotient space $\Gamma \backslash G$ is a closed manifold and we call $\Gamma \backslash G$ an infra-nilmanifold modeled on the Lie group $G$. If the holonomy group is trivial, we call $\Gamma \backslash G$ a nilmanifold. Every almost-Bieberbach group gives rise to a rational Mal'cev completion $N^\Q$ which corresponds to a rational Lie algebra $\lien$, equiped with a natural representation $H \to \Aut(N^\Q)$. 

Let $\alpha  \in \Aff(G)$ be an affine transformation satisfying $\alpha \Gamma \alpha^{-1} \subseteq \Gamma$. Then $\alpha$ induces a differentiable map $\bar{\alpha}$ on the infra-nilmanifold $\Gamma \backslash G$, given by $$\bar{\alpha}(\Gamma g) = \Gamma \hspace{0.5mm}{}^\alpha g.$$ The induced map $\bar{\alpha}$ is called an affine infra-nilmanifold endomorphism. The eigenvalues of $\alpha$ are defined as the eigenvalues of the linear part of $\alpha$, where eigenvalues of an automorphism of $G$ are the eigenvalues of the corresponding Lie algebra automorphism on the Lie algebra of $G$. If $\alpha \Gamma \alpha^{-1} = \Gamma$, then the map $\bar{\alpha}$ will be a diffeomorphism and we call $\bar{\alpha}$ an affine infra-nilmanifold automorphism. We call $\bar{\alpha}$ expanding if $\alpha$ only has eigenvalues $>1$ in absolute value and hyperbolic if it has no eigenvalue of absolute value $1$.

\smallskip

An expanding map $f: M \to M$ on a closed Riemannian manifold $M$ is a differentiable map such that there exists constants $c > 0$ and $\lambda > 1$ with $\Vert D f^n (v) \Vert \geq c \lambda^n \Vert v \Vert$ for all $v \in TM$ and all $n \geq 1$. By a result of Gromov, see \cite{grom81-1}, we know that every expanding map is topologically conjugate to an expanding affine infra-nilmanifold endomorphism. So up to homeomorphism, the infra-nilmanifolds are the only manifolds admitting an expanding map. 

Anosov diffeomorphisms are defined in a similar way, see \cite{smal67-1} and every hyperbolic affine infra-nilmanifold automorphism is an Anosov diffeomorphism. It follows from \cite[Corollary 3.5.]{deki99-1} that the existence of an Anosov diffeomorphism on a nilmanifold $N \backslash G$ depends only on the rational Mal'cev completion $N^\Q$. If $N\backslash G$ admits an Anosov diffeomorphism, we call the Lie algebra $\lien$ corresponding to $N^\Q$ Anosov. 

\smallskip

In \cite{dd14-1} is is showed that the existence of an expanding map depends only on the rational Mal'cev completion $N^\Q$ and the representation $H \to \Aut(N^\Q)$. By combining Theorem \ref{fieldalg} and Theorem \ref{holonomyalg}, we have a complete algebraic description of the infra-nilmanifolds admitting an expanding map:

\begin{Thm}
\label{classexp}
Let $\Gamma \backslash G$ be an infra-nilmanifold modeled on a Lie group $G$ with corresponding Lie algebra $\lie$. Then the following are equivalent:
\begin{enumerate}[$(1)$]
\item the infra-nilmanifold $\Gamma \backslash G$ admits an expanding map;
\item the Lie algebra $\lie$ has a positive grading;
\item the Lie algebra $\lie$ has an expanding automorphism.
\end{enumerate} 
\end{Thm}

As a corollary, we see that the existence of an expanding map depends only on the covering Lie group $G$:

\begin{Cor}
Let $M_1$ and $M_2$ be two infra-nilmanifolds modeled on the same Lie group. Then $M_1$ admits an expanding map if and only if $M_2$ admits an expanding map.
\end{Cor}

This is very different from the situation for Anosov diffeomorphisms, where the existence depends on the rational form $\lien$ and not only on the covering Lie group $G$. 
\begin{Ex}
Let $R$ be any commutative ring with unity and consider the Heisenberg group $$H_3(R) = \left\{ \begin{pmatrix} 1 & x & z \\0 & 1 & y\\0 & 0 & 1\end{pmatrix} \suchthat x, y, z \in R \right\}$$ over the ring $R$. Take the lattice $\Gamma = H_3(\Z) \oplus H_3(\Z)$ in the Lie group $G = H_3(\R) \oplus H_3(\R)$. Then the nilmanifold $\Gamma \backslash G$ does not admit an Anosov diffeomorphism, but there are nilmanifolds modeled on $G$ which do admit an Anosov diffeomorphism, see \cite{malf97-3}.
\end{Ex}

Similarly as for expanding maps, we can combine Theorem \ref{fieldalg2} and Theorem \ref{holonomyalg2} with the result \cite[Theorem 5.4.]{dd14-1} to find a complete algebraic characterization of the infra-nilmanifolds admitting a non-trivial self-cover, i.e. a self-cover which is not a homeomorphism:

\begin{Thm}
\label{classcoh}
Let $\Gamma \backslash G$ be an infra-nilmanifold modeled on a Lie group $G$ with corresponding Lie algebra $\lie$. Then the following are equivalent:
\begin{enumerate}
\item $\Gamma \backslash G$ admits a non-trivial self-cover;
\item $\lie$ has a non-negative and non-trivial grading;
\item $\lie$ has a partially expanding automorphism.
\end{enumerate} 
\end{Thm}

\noindent So also the existence of a non-trivial self-cover depends only on the covering Lie group $G$.

\smallskip

Another application of Theorem \ref{holonomy} is the following result:

\begin{Thm}
\label{commut}
Let $\Gamma \backslash G$ be an infra-nilmanifold admitting an expanding map and an Anosov diffeomorphism. Then there exists an expanding map and an Anosov diffeomorphism on $\Gamma \backslash G$ which commute.
\end{Thm}

\begin{proof}
If $\Gamma \backslash G$ admits an Anosov diffeomorphism, then \cite[Theorem A]{dv08-1} implies that there exists a hyperbolic automorphism $\varphi$ of $\lien$, i.e.\ with no eigenvalues of absolute value $1$, such that $\varphi$ has characteristic polynomial in $\Z[X]$,\hspace{0.2mm} $\vert \det(\varphi) \vert = 1$ and $\varphi$ commutes with every element of $H \subseteq \Aut(\lien)$. We can also assume that $\varphi$ is semisimple because of the multiplicative Jordan decomposition. Take $T$ the smallest linear algebraic subgroup of $\Aut(\lien^\C)$ which contains $\varphi$. The group $T$ will also commute with every element of $H$ and the group $H T$ forms a reductive subgroup of $\Aut(\lien^\C)$. By Theorem \ref{holonomy}, we conclude that there exists a positive grading $\lien = \bigoplus_{i > 0} \lien_i$ on $\Aut(\lien)$ which is preserved by $\varphi$ and every element of $H$. 

From the proof of \cite[Theorem A]{dv08-1}, it follows that some power $\varphi^k$ of $\varphi$ satisfies $\varphi^k \Gamma \varphi^{-k} = \Gamma$ and thus induces a hyperbolic affine infra-nilmanifold automorphism on $\Gamma \backslash G$. Let $p$ be any prime and consider the expanding automorphisms $\psi_p: \lien \to \lien$ which are given by $\psi_p(x) = p^i x$ for all $x \in \lien_i$. By the proof of \cite[Theorem 4.2.]{dd14-1} we know that there exists a prime $p$ and $l > 0$ such that $\psi_{p}^l \Gamma \psi_{p}^{-l} \subseteq \Gamma$. The expanding affine infra-nilmanifold endomorphism induced by $\psi_{p}^l$ commutes with the Anosov diffeomorphisms induced by $\varphi^k$ and this ends the proof. 
\end{proof}

\noindent Theorem \ref{commut} is also true for non-trivial self-covers on infra-nilmanifolds by replacing positive grading by non-trivial and non-negative grading in the proof.

\section{New type of example}

As another application of Theorem \ref{holonomy}, we construct a nilmanifold which admits an Anosov diffeomorphism but no non-trivial self-cover and so also no expanding map. This is the first example of a nilmanifold satisfying these properties. On the other hand, it is easy to give examples of nilmanifolds admitting an expanding map but no Anosov diffeomorphism, e.g.~the nilmanifold $H_3(\Z) \backslash H_3(\R)$.

\smallskip

First, we will construct a rational Lie algebra $\lien$ as a quotient of a free nilpotent Lie algebra $\lie$ by using the Hall basis of such a Lie algebra. Next, we give a general way of constructing automorphisms on this Lie algebra $\lien$. In this way, we give a finite group of automorphisms $H$ such that there exist no partially expanding automorphism of $\lien$ commuting with $H$. By Theorem \ref{field} and Theorem \ref{holonomy} this shows that there are no partially expanding automorphisms on this Lie algebra nor on any rational form of the Lie algebra $\lien^\R$. Finally, we use the techniques of \cite{dere13-1} to prove that the Lie algebra $\lien^\R$ has a rational form which is Anosov. So every nilmanifold corresponding to this rational form then has the desired properties.

\smallskip

\paragraph{\textbf{Hall basis:}}

Let $\lie$ be the free $6$-step nilpotent Lie algebra over $\Q$ on $4$ generators. Denote by $X_1, X_2, X_3, X_4$ a set of generators for the Lie algebra $\lie$. Consider the natural grading $\lie_1 \oplus \ldots \oplus \lie_6$ for $\lie$, where $\lie_1$ is the vector space spanned by $X_1, \ldots, X_4$. We say that the vector $a$ has degree $i$ if $a \in \lie_i$ and we denote this as $\deg(a) = i$. We will use the shorthand notation $[a,b,c]$ with $a,b,c \in \lie$ for the Lie bracket $[a,[b,c]]$ and similarly for longer brackets. If $Y$ and $Z$ are subsets of $\lie$, then we write $[Y,Z]$ for the Lie algebra generated by all elements $[y,z]$ with $y \in Y, z \in Z$. The Lie algebra $\lien$ we construct will be a double quotient of this free Lie algebra $\lie$. 

A explicit basis for the Lie algebra $\lie$ as vector space over $\Q$ is the Hall basis. The elements of this basis are constructed inductively: given the basis for $\lie_i$ with $1 \leq i \leq k$, we build the basis for $\lie_{k+1}$. We fix an order relation on the basis vectors of $\lie_1, \ldots, \lie_k$, assuming that $a < b$ if $\deg(a) < \deg(b)$.  The basis vectors for $\lie_{k+1}$ are then given by Lie brackets $[a,b] \in \lie_{k+1}$ with $\deg(a) + \deg(b) = k+1$ and $a < b$ with the extra condition that if $b = [b_1,b_2]$ then $a \geq b_1$. For more details and a proof that these vectors indeed form a basis, we refer to \cite{hall50-1}. In our case we will always assume that the order relation satisfies $X_1 < X_2 < X_3 < X_4$ on $\lie_1$, so $X_{i_1} < X_{i_2}$ if and only if $i_1 < i_2$. From now on we will fix a Hall basis for $\lie$ and an order relation on the basis vectors. 

Denote by $$\lie^\prime = [\lie,\lie] = \lie_2 \oplus \ldots \oplus \lie_6$$ the derived subalgebra of $\lie$ and take the subalgebra $\liem = [\lie^\prime, \lie^\prime] \subseteq \lie_4 \oplus \lie_5 \oplus \lie_6$ which is an ideal of $\lie$. It is well known that the elements $[X_{i_1}, \ldots, X_{i_k}]$ with $i_1 \geq \ldots \geq i_{k-1} < i_k$ of the Hall basis for $\lie$ project to a basis for the quotient $\faktor{\lie}{\liem}$. Therefore all the other elements of the Hall basis form a basis for $\liem$ as a vector space.
\begin{Lem}
\label{technical2}
Let $\beta$, $\gamma$ and $\delta$ be the Hall bases of $\lie_2$, $\lie_3$ and $\lie_4$ respectively. Then the Hall basis of $\lie_6 \cap \liem$ is given by the set $B = B_1 \cup B_2$ where 
\begin{align*}
B_1 = &\big\{ \left[b,d\right] \mid  b \in \beta, \hspace{1mm} d \in \delta, \hspace{1mm} d \neq [b_1,b_2] \text{ with } b_i \in \beta, b_1 > b \big\} \\ 
B_2 = &\big\{[c_1,c_2] \mid c_i \in \gamma, \hspace{1mm} c_1 < c_2 \big\}.\end{align*}
\end{Lem}  

\noindent Note that if $b < b_1 < b_2$, then the Jacobi identity gives us $$[b,b_1,b_2] = [b_1,b,b_2] - [b_2,b,b_1]$$ and the last two vectors are elements of the Hall basis. This implies that every element of $[\lie_2,\lie_4]$ can be written as a linear combination of elements in $B_1$. Similarly every element of $[\lie_3,\lie_3]$ can be written as a linear combination of elements in $B_2$.

\begin{proof}
The only other possibility for basis vectors in $\lie_6 \cap \liem$ is $[X_i,e]$ with $e$ in the Hall basis for $\lie_5$. Note that $e$ is not given by the Lie bracket between two vectors of degree $2$ and $3$, since $[X_i,e]$ is in the Hall basis. So $e$ is of the form $[X_j,d]$ with $d \in \delta$. Again, the vector $d$ is not equal to the Lie bracket of two vectors of degree $2$ since $[X_j,d]$ is in the Hall basis. So $d$ is of the form $d = [X_k,c]$ with $c \in \gamma$. But this implies that $[X_i,e] = [X_i,X_j,X_k,c] \notin \liem$. 
\end{proof}

\paragraph{\textbf{Construction of the Lie algebra $\lien$:}} 

Let $\lie$ be the free nilpotent Lie algebra on $4$ generators over $\Q$ as in the previous paragraph. Every permutation $s \in S_4$ determines an automorphism $\varphi_s$ of $\lie$ which is given by the relations $\varphi_s(X_i) = X_{s(i)}$ on the generators. Consider the automorphism $\alpha \in \Aut(\lie)$ of order $4$ which is induced by the permutation $(1234) \in S_4$. 

Let $I$ be the smallest ideal of $\lie$ such that \begin{align*} [X_i,X_1,X_3], [X_i,X_2,X_4] \text{ and } [X_{i_1},X_{i_2},X_{i_3},X_{i_4}] \in I \end{align*} for all $i, i_j \in \{1,\ldots,4\}$ with the indices $i_j$ distinct. From the Jacobi identity it follows that also $$[[X_{i_1},X_{i_2}],X_{i_3},X_{i_4}] = -[X_{i_3},X_{i_4},X_{i_1},X_{i_2}] + [X_{i_4},X_{i_3},X_{i_1},X_{i_2}]\in I$$ for all distinct $i_j \in \{1,\ldots,4\}$. The ideal $I$ is a graded ideal, meaning that $I$ is the direct sum of the vector spaces $I_k = I \cap \lie_k$. Let $\tilde{\lien}$ be the quotient Lie algebra $\faktor{\lie}{I}$ and denote by $\tilde{p}: \lie \to \tilde{\lien}$ the natural projection map. The ideal $I$ is invariant under $\alpha$, i.e. $\alpha(I) = I$ and hence $\alpha$ induces a map $\tilde{\alpha}: \tilde{\lien} \to \tilde{\lien}$.

Consider the vector $ v = [[X_4,X_3,X_4],X_2,X_1,X_2] \in \lie_6$ and write $\tilde{p}(v) = w$. The vector space spanned by $\tilde{p}\left(v - [X_2,X_4]\right) = w - \tilde{p}([X_2,X_4]) \in \tilde{\lien}$ is an ideal by definition of the Lie algebra $\tilde{\lien}$. Let $J$ be the smallest ideal of $\tilde{\lien}$ containing this vector and which is invariant under $\tilde{\alpha}$. Since $\tilde{\alpha}^2(v) = - v$ the dimension of $J$ is equal to $2$. Consider the Lie algebra $\lien = \faktor{\tilde{\lien}}{J}$ with projection map $p: \lie \to \lien$. The automorphism $\tilde{\alpha}$ induces an automorphism $\bar{\alpha} \in \Aut(\lien)$ of order $4$.

We start by showing that $\tilde{p}(v) = w \neq 0$.

\begin{Lem}
\label{technical}
The vector $v$ satisfies $\tilde{p}(v) \neq 0$ for $\tilde{p}: \lie \to \tilde{\lien}$ the projection map as above.
\end{Lem}

\begin{proof}
Since $v \in \liem \cap \lie_6$, it suffices to show that $v \notin I_6 \cap \liem$. We express the generators of $I_6 \cap \liem$ in terms of the Hall basis $B = B_1 \cup B_2$ of $\lie_6 \cap \liem$ as explained in Lemma \ref{technical2}. Note that $v$ or $-v$ is an element of $B_2$.

From the Jacobi identity, it follows that the vector space $I_6 \cap \liem$ satisfies $$I_6 \cap \liem =  [\lie_2,I_4] + [\lie_3,I_3] +[\lie_1,\lie_1,I_4 \cap \liem].$$ The elements of $[\lie_2,I_4] $ are linear combinations of elements in $B_1$. Let $\gamma$ be the Hall basis for $\lie_3$, then the vector space $[\lie_3,I_3]$ is spanned by vectors of the form $[c,[X_i,X_1,X_3]]$ and $[c,[X_i,X_2,X_4]]$ for $c \in \gamma$ and $1 \leq i \leq 4$. These vectors can easily be expressed in terms of the elements of $B_2$. This already implies that $v \notin [\lie_3,I_3]$.

To describe the generators of $ [\lie_1,\lie_1,I_4 \cap \liem]$ in the Hall basis, note that $$[X_j,b_1,b_2] =[b_1,X_j,b_2] - [b_2,X_j,b_1] $$ and thus $$[X_i,X_j,b_1,b_2] = [b_1,X_i,X_j,b_2] + [[X_i,b_1],[X_j,b_2]]  - [[X_i,b_2],[X_j,b_1]] - [b_2,X_i,X_j,b_1]$$ for all $[b_1,b_2] \in I_4$. By expressing the vectors $[b_1,X_i,X_j,b_2]$ and $[b_2,X_i,X_j,b_1]$ in the Hall basis, we only get elements of $B_1$ by using the remark under Lemma \ref{technical2}. On the other hand, the vector $[[X_i,b_1],[X_j,b_2]]  - [[X_i,b_2],[X_j,b_1]]$ is expressed only in terms of vectors in $X_2$. The only way of getting the vector $\pm v$ in this expression is in the case where $i=4, j=2, b_1 = \pm [X_3,X_4]$ and $b_2 = \pm [X_1,X_2]$ or in the situation with $i,j$ and $b_1, b_2$ interchanged. In both cases, the other vector is up to sign equal to $$[[X_2,X_3,X_4],X_4,X_1,X_2] = [[X_3,X_2,X_4],X_4,X_1,X_2] - [[X_4,X_2,X_3],X_4,X_1,X_2],$$ where these last two vectors are up to sign in $X_2$. The first vector lies in $[\lie_3,I_3]$, but the second vector $[[X_4,X_2,X_3],X_4,X_1,X_2]$ is not an element of $[\lie_3,I_3]$. This last statement can be checked by using the expression of the generators of $[\lie_3,I_3]$ in terms of the basis $B_2$. We conclude that $v \notin I$. 
\end{proof}

From Lemma \ref{technical} it follows that $p(v) \neq 0$. 

\begin{Lem}
\label{veelwerk}
Let $v = [[X_4,X_3,X_4],X_2,X_1,X_2] \in \lie$ and $\lien$ the Lie algebra as above with projection map $p: \lie \to \lien$. Then $p(v) \neq 0$.
\end{Lem}
\begin{proof}
The vector $\tilde{p}(v) \neq 0$ by Lemma \ref{technical}. Therefore the vectors $\tilde{p}(v), \tilde{p}([X_1,X_3])$ and $\tilde{p}([X_2,X_4])$ are linearly independent in $\tilde{\lien}$ since $I$ is a graded ideal and $I_2 = I \cap \lie_2 = 0$. Assume that $p(v) = 0$, then the ideal $J$ contains the vectors $\tilde{p}(v), \tilde{p}([X_1,X_3])$ and $\tilde{p}([X_2,X_4])$. This is impossible since the dimension of $J$ is equal to $2$. 
\end{proof}

\paragraph{\textbf{Automorphisms on $\lien$:}} 

By the explicit construction of the Lie algebra $\lien$ as a quotient of the free Lie algebra $\lie$ we can give a general way of constructing automorphisms on $\lien^E$ for any field extension $E \supseteq \Q$. Consider the linear subspace $\lie_1^E$ of $\lie^E$ spanned by $X_1, \ldots, X_4$. Take $\lambda_1, \ldots, \lambda_4 \in E$ such that $\lambda_1 \lambda_2 \lambda_3 \lambda_4 = 1$ and consider the linear map $\lie_1^E \to \lie_1^E$ given by $$X_i \mapsto \lambda_i X_i.$$ This map uniquely extends to an automorphism $\varphi: \lie^E \to \lie^E$ and it is easy to see that it also induces an automorphism $\bar{\varphi}$ on the Lie algebra $\lien^E$. 

Take the basis $\bar{X}_1, \ldots, \bar{X}_4$ for the vector space $\faktor{\lien^E}{[\lien^E,\lien^E]}$. Under the natural projection map $\pi: \Aut(\lien^E) \to \Aut \left(\faktor{\lien^E}{[\lien^E,\lien^E]}\right) \cong \GL(4,E)$, the automorphism $\bar{\varphi}$ is mapped to the diagonal matrix with eigenvalues $\lambda_1, \lambda_2, \lambda_3$ and $\lambda_4$. This will be an important construction of automorphisms on the Lie algebra $\lien^E$.

\smallskip

As a consequence of this construction for automorphisms, we have the following proposition:

\begin{Prop}
\label{proppartexp}
The Lie algebra $\lien$ has no partially expanding automorphisms.
\end{Prop}

\begin{proof}

Let $H$ be the subgroup of $\GL(4,\Q)$ generated by the matrices $$\begin{pmatrix}1 & 0 & 0 & 0 \\0 & 1 & 0 & 0 \\ 0 & 0 & -1 & 0 \\0 & 0 & 0 & -1\end{pmatrix} \text{ and } \begin{pmatrix}1 & 0 & 0 & 0 \\0 & -1 & 0 & 0 \\ 0 & 0 & 1 & 0 \\0 & 0 & 0 & -1\end{pmatrix}.$$ This subgroup is isomorphic to $\Z_2 \oplus \Z_2$ and the centralizer of $H$ in $\GL(4,\Q)$ is given by the diagonal matrices $D(n,\Q)$. As described just above this theorem, each of the generators of $H$ above induces an automorphism of the Lie algebra $\lien$ and thus we get a faithful representation $i: H \to \Aut(\lien)$. 

Assume that $\lien$ does have a partially expanding automorphism $\varphi$. By Theorem \ref{holonomy}, we can assume that $\varphi$ commutes with every element of the finite group $i(H)$. Consider the vector space $\faktor{\lien}{[\lien,\lien]}$ with basis $\bar{X}_1, \ldots, \bar{X}_4$ and the natural projection map $\pi: \Aut(\lien) \to \Aut \left(\faktor{\lien}{[\lien,\lien]}\right) \cong \GL(4,\Q)$. Since $\pi(\varphi)$ lies in the centralizer of $H$, we know that $$\pi(\varphi) = \begin{pmatrix} \lambda_1 & 0 & 0 & 0 \\0 & \lambda_2 & 0 & 0 \\ 0 & 0 & \lambda_3 & 0\\0 & 0 & 0 &  \lambda_4\end{pmatrix}.$$ The vector $p(v)$, with $v$ as in the definition of $\lien$ above, is then an eigenvector of $\varphi$ with eigenvalue $\lambda_1 \lambda_2^2 \lambda_3 \lambda_4^2$ and $p([X_2,X_4])$ is an eigenvector of eigenvalue $\lambda_2 \lambda_4$. Since $p(v) = p([X_2,X_4]) \neq 0$ by Lemma \ref{veelwerk}, it must hold that $\lambda_1 \lambda_2 \lambda_3 \lambda_4 = 1$. This is a contradiction since $\varphi$ is a partially expanding automorphism.
\end{proof}

\paragraph{\textbf{Anosov Lie algebra:}} A rational form of a real Lie algebra $\lien$ is a rational subalgebra $\liem \subseteq \lien$ such that $\liem^\R = \R \otimes \liem = \lien$. The paper \cite{dere13-1} gives a general way of constructing rational forms of real Lie algebras which are Anosov. In particular, we will use \cite[Corollary 2.7]{dere13-1} to show that the real Lie algebra $\lien$ has a rational form which is Anosov:




\begin{Prop}
The Lie algebra $\lien^\R$ has a rational form which is Anosov.
\end{Prop}

\begin{proof}
Take $\Q \subseteq E \subseteq \R$ a field extension with Galois group $\Gal(E,\Q) \cong \Z_4$ and denote by $\sigma$ a generator of $\Gal(E,\Q)$. Let $\mu$ be a unit Pisot number in $E$ and write $\mu_i = \sigma^{i-1}(\mu)$. By squaring $\mu$ if necessary, we can assume that $\mu_1 \mu_2 \mu_3 \mu_4=1.$ Take $\varphi: \lien \to \lien$ the automorphism induced by the linear map that maps $$X_i \to \mu_i X_i,$$ as explained above Proposition \ref{proppartexp}.

All eigenvalues of the automorphism $\varphi$ are products of the algebraic units $\mu_i$ of length at most $6$. Note that $\mu$ satisfies the full rank condition (see \cite{dere13-1}), meaning that if  $$\prod_{j=1}^4 \mu_j^{d_j} = \pm 1,$$ it holds that $d_1 = d_2 = d_3 = d_4$. Therefore, the only possibility to get an eigenvalue of absolute value $1$ is $\mu_1 \mu_2 \mu_3 \mu_4$. By construction of the Lie algebra $\tilde{\lien}$, all the eigenvectors with eigenvalue $\mu_1 \mu_2 \mu_3 \mu_4$ lie in $I$, so this eigenvalue does not occur. Hence, $\varphi$ has no eigenvalues of absolute value $1$. 

Consider the representation $\rho: \Gal(E,\Q) \to \Aut(\lien)$ given by $\rho(\sigma) = \bar{\alpha}$. The maps $\rho$ and $\varphi$ satisfy the conditions of \cite[Corollary 2.7]{dere13-1}, where the subspaces $V_\lambda$ are the eigenspaces of the map $\varphi$ and the map $f$ of the theorem is equal to $\varphi$. This implies that $\lien^E$ (and therefore also $\lien^\R$) has a rational form which is Anosov. 
\end{proof}

\bibliography{ref}
\bibliographystyle{plain}

\end{document}